\documentclass[10pt]{article}

\usepackage{authblk}

\usepackage{amsmath}
\usepackage{amsthm}
\usepackage{amsfonts}
\usepackage{amssymb}
\usepackage[all]{xy}
\usepackage{url}

\DeclareMathSymbol{\rightrightarrows}  {\mathrel}{AMSa}{"13}

\def\sd{\operatorname{sd}}

\def\Ex{\operatorname{Ex}}

\catcode`\@=11
\def\varholim@#1#2{\mathop{\vtop{\ialign{##\crcr
 \hfil$#1\m@th\operator@font holim$\hfil\crcr
 \noalign{\nointerlineskip\kern\ex@}#2#1\crcr
 \noalign{\nointerlineskip\kern-\ex@}\crcr}}}}
\def\hocolim{\mathpalette\varholim@\rightarrowfill@} 
\def\hoinvlim{\mathpalette\varholim@\leftarrowfill@}
\catcode`\@=\active

\newtheorem{theorem}{Theorem}%[section]
\newtheorem{lemma}[theorem]{Lemma}
\newtheorem{corollary}[theorem]{Corollary}

\newtheorem{proposition}[theorem]{Proposition}

\theoremstyle{definition}

\newtheorem{remark}[theorem]{Remark}

\begin{document}

\title{Persistent homotopy theory}

\author{J.F. Jardine\thanks{Supported by NSERC.}}

\affil{\small Department of Mathematics\\University of Western Ontario\\London, Ontario, Canada
}
\affil{jardine@uwo.ca}

\renewcommand{\thefootnote}{\fnsymbol{footnote}} 
\renewcommand{\thefootnote}{\arabic{footnote}} 

%\subjclass[2010]{Primary 55U10; Secondary 62R40, 68P05}

%55U10 simplicial sets and complexes
%62R40 topological data analysis (stats)
%68P05 data structures (CS)
%68T09 computational aspects of data analysis and big data (CS)

%\date{April 30, 2020}

\maketitle

\begin{abstract}
Vietoris-Rips and degree Rips complexes are represented as homotopy types by their underlying posets of simplices, and basic homotopy stability theorems are recast in these terms. These homotopy types are viewed as systems (or functors), which are defined on a parameter space. The category of systems of spaces admits a partial homotopy theory that is based on controlled equivalences, suitably defined, that are the output of homotopy stability results. 
  \end{abstract}

\section*{Introduction}

A prototypical homotopy stability result asserts that, if one adds points to a data set $X$ that are close in a suitable sense to form a new data set $Y$, then the corresponding inclusion $V_{\ast}(X) \to V_{\ast}(Y)$ of Vietoris-Rips systems is a strong deformation retract up to a bounded shift, where the bound depends linearly on how close the points of $Y$ are to points of $X$.

The language in this last paragraph is a bit colloquial, and it involves new terms that need to be explained. In particular, a system of spaces $X$ is a functor $s \mapsto X_{s}$ where $s$ is a member of the real parameter poset $[0,\infty)$ and each $X_{s}$ is a ``space'' or simplicial set, while a map of systems is a natural transformation of functors.

For a finite metric space $X$ (a data set), the Vietoris-Rips complexes $s \mapsto V_{s}(X)$ form such a system, and an inclusion of finite metric spaces $X \subset Y$ induces a transformation $V_{s}(X) \to V_{s}(Y)$ that is natural in the distance parameter $s$. Recall that $V_{s}(X)$ is the finite simplicial complex whose simplices are subsets $\sigma$ of $X$ such that the distance $d(x,y) \leq s$ for all $x,y \in \sigma$. 

The Vietoris-Rips complex $V_{s}(X)$ is defined as an abstract simplicial complex, and one usually makes it into a space by constructing its realization. An alternative is to put a total order on the vertices (which is consistent with listing the data set $X$), and then form an associated simplicial set as a subobject of a simplex that is determined by the order on $X$. This simplicial set also has a realization, which is homeomorphic to the realization of the abstract simplicial complex. Both routes lead to the same space, and hence represent the same homotopy type.
\medskip

There is a different approach. The basic method of this paper is to treat the poset $P_{s}(X)$ of simplices of $V_{s}(X)$ as a homotopy theoretic object in its own right by using the nerve $BP_{s}(X)$ of $P_{s}(X)$. The space $BP_{s}(X)$ is the barycentric subdivision of the Vietoris-Rips complex $V_{s}(X)$, and therefore has the same homotopy type.

This construction may seem fraught with complexity, but one can restrict to low dimensional simplices as necessary.
The advantage of the poset approach is that the nerves of the posets $P_{s}(X)$ can be employed to great theoretical effect, by using basic features of Quillen's theory of homotopy types of posets \cite{Q4}.

For example, suppose that $Y$ is a finite metric space, and that $X$ is a subset of $Y$. Suppose that $r \geq 0$ is a real parameter and that for each $y \in Y$ there is an $x \in X$ such that $d(x,y) < r$, where $d$ is the metric on $Y$. Then one constructs a retraction function $\theta: Y \to X$ by insisting that $\theta(y)$ is a point of $X$ such that $d(y,\theta(y)) < r$. It follows from the triangle identity shows the function $\theta$ induces a poset morphism $\theta: P_{s}(Y) \to P_{s+2r}(X)$, and there is a diagram of morphisms
\begin{equation}\label{eq 1}
\xymatrix{
P_{s}(X) \ar[r] \ar[d]_{i} & P_{s+2r}(X) \ar[d]^{i} \\
P_{s}(Y) \ar[r] \ar[ur]^{\theta} & P_{s+2r}(Y)
}
\end{equation}
in which the upper triangle commutes on the nose, and the bottom triangle commutes up to a natural transformation that fixes $P_{s}(X)$. The horizontal and vertical morphisms are the natural inclusions.

This construction translates directly to a proof of the Rips stability theorem after applying the nerve functor --- this is Theorem \ref{th 4} below.
There is a corresponding construction and result for the degree Rips filtration (Theorem \ref{th 6}), where one uses a more interesting distance criterion that involves configuration spaces.
We also present, in Theorem \ref{th 5}, a quick proof of the version of the Rips stability theorem given by Blumberg-Lesnick \cite{BlumLes} that uses only poset techniques.

These results are proved in Section 2. The basic terminology appears in Section 1, along with a relatively simple model for the fundmental groupoid of the space $BP_{s}(X)$. 
\medskip

The diagram
(\ref{eq 1}) is a ``homotopy interleaving'', and is a strong deformation retraction up to a shift --- in this case the shift is $2r$. Its existence  implies that there is a commutative diagram
\begin{equation*}
\xymatrix{
\pi_{n}(BP_{s}(X),x) \ar[r] \ar[d]_{i_{\ast}} & \pi_{n}(BP_{s+2r}(X),x) \ar[d]^{i_{\ast}} \\
\pi_{n}(BP_{s}(Y),x) \ar[r] \ar[ur]^{\theta} & \pi_{n}(BP_{s+2r}(Y),x)
}
\end{equation*}
of maps between homotopy groups for each choice of base point $x \in X$. There are similar induced diagrams in path components and in homology groups.

It follows that if $\alpha \in \pi_{n}(BP_{s}(X),x)$ maps to $0 \in \pi_{n}(BP_{s}(Y),x)$, then $\alpha$ maps to $0$ in $\pi_{n}(BP_{s+2r}(X),x)$, so that the vertical maps $i_{\ast}$ are $2r$-mono\-morphisms, suitably defined. Similarly, the maps $i_{\ast}$ are $2r$-epimorphisms, in that every $\beta \in  \pi_{n}(BP_{s}(Y),x)$ maps to an element of $\pi_{n}(BP_{s+2r}(Y),x)$ which is in the image of the homomorphism $i_{\ast}$. 

The maps $i_{\ast}: \pi_{n}(BP_{s}(X),x) \to \pi_{n}(BP(Y),x)$ are $2r$-isomorphisms in the sense that they are $2r$-mono\-mor\-phisms and a $2r$-epimorphisms. A similar observation holds for path components, and one says that the $2r$-interleaving produced by the Rips stability theorem is a $2r$-equivalence.

More generally, one defines families of $r$-equivalences of systems of spaces for all $r \geq 0$, and
a map $X \to Y$ of systems of spaces is a {\it controlled equivalence} if it is an $r$-equivalence for some $r \geq 0$.
\medskip

The third section of this paper is a general study of controlled equivalences of systems of spaces, along with their interactions with sectionwise fibrations and sectionwise cofibrations.

A map of systems $f: X \to Y$ is a sectionwise fibration if all of its constituent maps $f: X_{s} \to Y_{s}$ are fibrations of simplicial sets. Sectionwise cofibrations and sectionwise weak equivalences are defined analogously. 

Quillen's triangle axiom {\bf CM2} does not hold for the class of $r$-equivalences, but a modification is possible: Lemma \ref{lem 12} implies, for example, that if $f: X \to Y$ is an $s$-equivalence and $g: Y \to Z$ is an $r$-equivalence, then the composite $g \cdot f: X \to Z$ is an $(r+s)$-equivalence. 

It is shown, in a sequence of lemmas leading to Theorem \ref{th 15}, that maps $p: X \to Y$ which are both sectionwise fibrations and $r$-equivalences pull back to maps which are sectionwise fibrations and $2r$-equivalences. The doubling of the parameter from $r$ to $2r$ reflects the usual two obstructions in the argument for the corresponding classical result for simplicial sets.

It is tempting to think that Theorem \ref{th 15} has a dual formulation that holds for sectionwise cofibrations, but such a result has not been proved.

We still have partial glueing results. Perhaps most usefully, if there is a pushout diagram of systems
\begin{equation*}
  \xymatrix{
    A \ar[r] \ar[d]_{i} & C \ar[d]^{i_{\ast}} \\
    B \ar[r] & D
  }
  \end{equation*}
where $i$ is a sectionwise cofibration, then $i_{\ast}$ is a sectionwise cofibration, and the following statements hold:
\begin{itemize}
\item[1)] If the map $\pi_{0}A \to \pi_{0}B$ is an $r$-isomorphism then the map $i_{\ast}: \pi_{0}C \to \pi_{0}D$ is an $r$-isomorphism.
\item[2)] If the maps $H_{k}(A) \to H_{k}(B)$ are $r$-isomorphisms in homology (arbitrary coefficients) for $k \geq 0$, then the maps $H_{k}(C) \to H_{k}(D)$ are $2r$-isomorphisms.
\end{itemize}
These statements are proved in Lemma \ref{lem 18} of this paper --- the arguments are not difficult.

We also show, in Lemma \ref{lem 17}, that if $i: A \to B$ is an $r$-interleaving, or a strong deformation retraction up to shift $r$, then the same holds for the map $i_{\ast}: C \to D$.

This applies in particular to cofibrations that arise from stability theorems. Thus, if $i$ is a map $BP_{\ast}(X) \to BP_{\ast}(Y)$ that is associated to an inclusion $X \subset Y$ of finite metric spaces that satisfies $d_{H}(X,Y) < r$, then the map $i_{\ast}: C \to D$ is a strong deformation retraction up to shift $2r$.

\tableofcontents

\section{Posets}

A {\it data set} $X$ is a finite subset of a metric space $Z$. The collection of data sets in $Z$ with inclusions between them forms a poset, which is denoted by $D(Z)$.

Suppose that $s \geq 0$, and that $X$ is a data set in $Z$. Write $P_{s}(X)$ for the poset of all subsets $\sigma \subset X$ such that $d(x,y) \leq s$ for all $x,y \in \sigma$.

  The poset $P_{s}(X)$ is the poset of simplices of the Vietoris-Rips complex $V_{s}(X)$ of $X$. The members $\sigma \subset X$ of $P_{s}(X)$ are simplices of dimension $n-1$, where $n = \vert \sigma \vert$ is the number of elements of $\sigma$.
  \medskip
  
Each poset $P_{s}(X)$ is a finite category. Other examples of finite posets are given by the finite ordinal numbers
\begin{equation*}
  \mathbf{n} = \{0,1, \dots ,n\},
\end{equation*}
with the obvious ordering.

There is a functorial method of assigning a simplicial set $BC$ to a small category $C$, where the $n$-simplices of $BC$ are the functors $\alpha: \mathbf{n} \to C$, or strings of composable morphisms in $C$ of length $n$. The simplicial structure maps of $BC$ are defined by composition with the functors (poset maps) $\mathbf{m} \to \mathbf{n}$ between finite ordinal numbers. The simplicial set $BC$ is variously called the {\it nerve} or the {\it classifying space} of $C$.

A group $G$ is a category (groupoid) with one object, and $BG$ is a model for the classifying space of $G$.

The nerve construction can also be applied to the ordinal number posets $\mathbf{n}$, and there is a natural isomorphism
\begin{equation*}
  B\mathbf{n} \cong \Delta^{n},
  \end{equation*}
where $\Delta^{n}$ is the standard $n$-simplex in simplicial sets.

The nerve functor $C \mapsto BC$ also preserves products, so that there is an isomorphism
\begin{equation*}
  B(C \times \mathbf{1}) \cong BC \times \Delta^{1}.
\end{equation*}
The existence of this isomorphism implies that the nerve functor takes natural transformations to simplicial homotopies.

It is standard to identify natural transformations with homotopies in this form of categorical homotopy theory.
\medskip

We have poset inclusions
  \begin{equation*}
    \sigma: P_{s}(X) \subset P_{t}(X), \enskip s \leq t,
  \end{equation*}
  for the data set $X$.
  
  Observe that $P_{0}(X)$ is the discrete poset (category) whose objects are the elements of $X$, and that $P_{t}(X)$ is the poset $\mathcal{P}(X)$ of all subsets of $X$ for $t$ sufficiently large.

  There is an isomorphism of posets
  \begin{equation*}
    \mathcal{P}(X) \cong \mathbf{1}^{\times m},
  \end{equation*}
  where $\mathbf{1}$ is the poset $\{0,1\}$ and $m$ is the cardinality of the set $X$. The isomorphism sends a subset $A$ of $X$ to the $m$-tuple $(\epsilon_{x})_{x \in X}$, where
  \begin{equation*}
    \epsilon_{x} =
    \begin{cases}
      1 & \text{if $x \in A$, and} \\
      0 & \text{if $x \notin A$.}
    \end{cases}
  \end{equation*}
  It follows that there is an isomorphism of simplicial sets
  \begin{equation*}
    B\mathcal{P}(X) \cong (\Delta^{1})^{\times m}.
  \end{equation*}
In particular, the simplicial set (or space) $BP_{t}(X)$ is contractible if $t$ is sufficiently large.
\medskip

\medskip

The Vietoris-Rips complex $V_{s}(X)$ is a finite abstract simplicial complex, and $P_{s}(X)$ is its poset of simplices.

The realization $\vert V_{s}(X) \vert$ of $V_{s}(X)$ is constructed, as a space, by glueing affine simplices together along face relationships, and it is standard to identify the simpilicial complex $V_{s}(X)$ with its realization. 

The nerve $BP_{s}(X)$ of the poset $P_{s}(X)$ is a simplicial set whose realization is the barycentric subdivision $\sd(V_{s}(X))$ of $V_{s}(X)$. The subdivision $\sd(V_{s}(X))$ is naturally weakly equivalent to $V_{s}(X)$ \cite[III.4]{GJ}, \cite{J34}.

There is a non-canonical method of associating a simplicial set structure to $V_{s}(X)$ that arises from a total ordering, or listing
\begin{equation*}
\phi: \mathbf{N} \xrightarrow{\cong} X
\end{equation*}
of the elements of the data set $X$. In the presence of such a listing, the set $X$ has $N+1$ elements, and the simplicial set $V_{s}(X)$ is the subcomplex of the standard simplex $\Delta^{N}$ whose non-degenerate simplices are the members of the original abstract simplicial complex. In this case, the simplicial set $V_{s}(X)$ is oriented by the total ordering $\phi$ on $X$. Its realization, as a simplicial set, is homeomorphic to the realization of the underlying abstract simplicial complex, so that its homotopy type is independent of the ordering.

The method of this paper is to identify the poset $P_{s}(X)$ with the homotopy type $BP_{s}(X)$ directly, without either constructing a realization or assuming a particular list structure on $X$. This is consistent with the general identification of small categories with homotopy types, which was pioneered by Quillen \cite{Q3}, \cite{Q4} during the early development of algebraic $K$-theory.
\medskip

Suppose that $k$ is a non-negative integer. The poset $P_{s}(X)$ has a subobject $P_{s,k}(X) \subset P_{s}(X)$, which is the subposet of simplices $\sigma$ such that each element $x \in \sigma$ has at least $k$ distinct ``neighbours'' $y$ in $X$ (not necessarily in $\sigma$) such that $d(x,y) \leq s$.

The poset $P_{s,k}(X)$ is the poset of simplices of the degree Rips complex (or Lesnick complex) $L_{s,k}(X)$.
\medskip

For $s \leq t$ we have a diagram of poset inclusions
\begin{equation*}
\xymatrix{
P_{s}(X) \ar[r]^{\sigma} & P_{t}(X) \\
P_{s,k}(X) \ar[u] \ar[r]^{\sigma} & P_{t,k}(X) \ar[u] \\
P_{s.k+1}(X) \ar[r]_{\sigma} \ar[u] & P_{t.k+1}(X) \ar[u]
}
\end{equation*}
The notation $\sigma$ will always be used for poset inclusions associated to changes of distance parameter.

Observe also that
\begin{itemize}
\item[1)]
$P_{s,0}(X) = P_{s}(X)$ for all $s$, and 
\item[2)]
$P_{s,k}(X) = \emptyset$ for $k$ sufficiently large.
\end{itemize}
The objects $P_{\ast,k}(X)$ form the {\it degree Rips filtration} of the Vietoris-Rips system of posets $P_{\ast}(X)$.
\medskip

The simplicial set $BP_{s}(X)$ is model for $V_{s}(X)$ in the homotopy category, but it may seem intractably large since all simplices of $V_{s}(X)$ are vertices of $BP_{s}(X)$, while one can only practically recover the low dimensional part of the simplicial structure of $V_{s}(X)$ in concrete examples. That said, one only needs low dimensional simplices to compute low dimensional homotopy or homology groups of $BP_{s}(X)$. This is illustrated as follows.
\medskip

Suppose, generally, that the poset $P$ is a subobject of a power set $\mathcal{P}(X)$ is a subposet that is closed under taking non-empty subsets, so that $P$ defines an abstract simplicial complex.

Suppose given a list $x_{0}, \dots ,x_{k}$ of elements of $X$ such that $d(x_{i},x_{j}) \leq s$. This list may have repeats, and can be viewed as a function $x: \{0,1,\dots ,k\} \to X$ which may not be injective.
Write
\begin{equation*}
  [x_{0}, \dots ,x_{k}] = \{x_{0}\} \cup \dots \cup \{x_{k}\}.
  \end{equation*}
in $X$. This set can be identified with the image of the function $x$.

There is a graph $Gr(P)$ whose vertices are the singleton elements (vertices) $\{x\}$ of $P$, and there is an edge $x \to y$ if $[x,y]$ is an object of $P$.

Observe that there is an edge $[x,y]: x \to y$ if and only if there is an edge $[y,x]:y \to x$, and there is an edge $[x,x]: x \to x$.

Write $\Gamma(P)$ for the category generated by $Gr(P)$, subject to relations defined by the simplices $[x_{0},x_{1},x_{2}]$. Then we have the following:

\begin{proposition}\label{prop 1}
  The category $\Gamma(P)$ is a groupoid, and there are equivalences
\begin{equation*}
  \Gamma(P) \simeq G(P) \simeq \pi(BP),
\end{equation*}
where $\pi(BP)$ is the fundamental groupoid of the space $BP$, and $G(P)$ is the free groupoid on the poset $P$.
\end{proposition}

The groupoid $G(P)$ can be identified up to natural equivalence with the fundamental groupoid $\pi (BP)$ by \cite[III.2.1]{GJ}. In more detail, $\pi BP$ is isomrphic to $G(P_{\ast}(BP))$ where $P_{\ast}(BP)$ is the path category of $BP$, and there is an isomorphism $P_{\ast}(BP) \cong P$, essentially by inspection (see also \cite{pathcat}).

\begin{corollary}
The category $\Gamma(P_{s}(X))$ is a groupoid, 
 and there are equivalences
\begin{equation*}
  \Gamma(P_{s}(X)) \simeq G(P_{s}(X)) \simeq \pi(BP_{s}(X)).
\end{equation*}
  \end{corollary}
There is an equivalence of groupoids $\pi(BP_{s}(X)) \simeq \pi(V_{s}(X))$, since the spaces $BP_{s}(X)$ and $V_{s}(X)$ are weakly equivalent, and so the fundamental groupoid $\pi(V_{s}(X))$ is weakly equivalent to $\Gamma(P_{s}(X))$.

\begin{proof}[Proof of Proposition \ref{prop 1}]
  We show that
  \begin{itemize}
    \item[1)]
      The category $\Gamma(P)$ is a groupoid.
\item[2)]
      There is an equivalence of groupoids $\Gamma(P) \simeq G(P)$.
  \end{itemize}
  
For the first claim, the edges $[x,x]$ represent $2$-sided identities, on account of the existence of the simplices $[x,x,y]$ and $[x,y,y]$. Then the simplices $[x,y,x]$ and $[y,x,y]$ are used to show that each edge $[x,y]$ represents an invertible morphism of $\Gamma_{s}(X)$.

For the second claim, pick an element $x_{\sigma} \in \sigma$ for each simplex $\sigma \in P$.

For $\sigma \subset \tau$ in $P$, the list $x_{\sigma},x_{\tau}$ consists of elements of $\tau$, so that $[x_{\sigma},x_{\tau}] \subset \tau$ is a simplex of $P$. If $\sigma \subset \tau \subset \gamma$ are morphisms of $P$, then $[x_{\sigma},x_{\tau},x_{\gamma}] \subset \gamma$ is a simplex of $P$, and so there is a commutative diagram
\begin{equation*}
\xymatrix{
x_{\sigma} \ar[r]^{[x_{\sigma},x_{\tau}]} \ar[dr]_{[x_{\sigma},x_{\gamma}]} & x_{\tau} \ar[d]^{[x_{\tau},x_{\gamma}]} \\
& x_{\gamma}
}
\end{equation*}
in $\Gamma(P)$. It follows that sending the morphism  $\sigma \subset \tau$ to
$[x_{\sigma},x_{\tau}]: x_{\sigma} \to x_{\tau}$ defines a functor $P \to \Gamma(P)$, which induces a functor
\begin{equation*}
\phi: G(P) \to \Gamma(P).
\end{equation*}

Suppose that $[x,y]: x \to y$ is an edge of the graph $Gr(P)$. Then the associated inclusions 
\begin{equation*}
\{ x \} \to [x,y] \leftarrow \{ y \}
\end{equation*}
in $P$ define a morphism $[x,y]_{\ast}: \{ x \} \to \{ y \}$ of $G(P)$.

If $[x,y,z]$ is a simplex of $P$ then the diagram of inclusions
\begin{equation*}
\xymatrix{
&& \{ y \} \ar[dl] \ar[d] \ar[dr] \\
&[x,y] \ar[r] & [x,y,z] & [y,z] \ar[l] \\
\{ x \} \ar[ur] \ar[urr] \ar[rr] && [x,z] \ar[u] && \{ z \}
\ar[ll] \ar[ull] \ar[ul]
}
\end{equation*} 
is used to show that $[y,z]_{\ast} \cdot [x,y]_{\ast} = [x,z]_{\ast}$ in the groupoid $G(P)$. 

It follows that the assignment that takes an edge $x \to y$ of $Gr(P)$ to the morphism $[x,y]_{\ast}: \{ x \} \to \{ y \}$ defines a functor
\begin{equation*}
\psi: \Gamma(P) \to G(P).
\end{equation*}

The inclusions $\{ x_{\sigma} \} \subset \sigma$ define a natural isomorphism 
\begin{equation*}
\psi \cdot \phi \xrightarrow{\cong} 1_{G(P)},
\end{equation*}
and
the composite $\phi \cdot \psi$ is the identity on $\Gamma(P)$.
\end{proof}

\section{Stability}

Suppose that $Z$ is a metric space, and let $D(Z)$ be the poset of finite subsets (data sets) in $Z$.

The poset $D(Z)$ has the {\it Hausdorff metric} $d_{H}$, which can be described heuristically, relative to a fixed $r \geq 0$, as follows:
  \begin{itemize}
\item[1)] Suppose that $X \subset Y$ in $D(Z)$. Then $d_{H}(X,Y) < r$ if for all $y \in Y$ there is an $x \in X$ such that $d(y,x) < r$.
\item[2)]
For arbitrary  $X,Y \in D(Z)$:\enskip  $d_{H}(X,Y) < r$ if and only if (equivalently)
\begin{itemize}
\item[a)]
  $d_{H}(X,X\cup Y) < r$ and $d_{H}(Y,X\cup Y) < r$.
\item[b)]
for all $x \in X$ there is a $y \in Y$ such that $d(x,y) < r$, and
for all $y \in Y$ there is an $x \in X$ such that $d(y,x) < r$.
\end{itemize}
\end{itemize}

  We also have the following:

  \begin{lemma}\label{lem 3}
Suppose that $X$ and $Y$ are data sets in a metric space $Z$, and suppose that $d_{H}(X \cap Y,X) < r$.Then $d_{H}(Y,X \cup Y) < r$.
  \end{lemma}

  Lemma \ref{lem 3} is easily proved. The statement can be visualized by the following diagram of labelled inclusions:
  \begin{equation*}
    \xymatrix{
      X \cap Y \ar[r] \ar[d]_{r} & Y \ar[d]^{r} \\
      X \ar[r] & X \cup Y
    }
  \end{equation*}
 \medskip 

Now suppose that $X \subset Y$ are data sets in $Z$, and supppose that $d_{H}(X,Y) < r$. Construct a function $\theta: Y \to X$ such that
  \begin{equation*}
    \theta(y) =
    \begin{cases}
      y & \text{if $y \in X$} \\
      x_{y} & \text{where  $x_{y} \in X$ with $d(y,x_{y}) < r$, if $y \notin X$}.
    \end{cases}
  \end{equation*}
  The function $\theta$ is a retraction and is not an inclusion. We are interested in the images $\theta(\tau)$ of subsets $\tau$ of $Y$.
  \medskip

  If $\tau \in P_{s}(Y)$ then $\theta(\tau) \in P_{s+2r}(X)$ by the triangle identity, and the assignment $\tau \mapsto \theta(\tau)$ respects inclusions of finite sets $\tau$.

  It follows that we have a diagram of poset morphisms
\begin{equation*}
  \xymatrix{
    P_{s}(X) \ar[r]^-{\sigma} \ar[d]_{i}  & P_{s+2r}(X) \ar[d]^{i} \\
    P_{s}(Y) \ar[r]_-{\sigma} \ar[ur]^(.4){\theta} & P_{s+2r}(Y)
  }
\end{equation*}
such that upper triangle commutes, and lower triangle commutes up to homotopy, in the sense that there are inclusions
\begin{equation}\label{eq 2}
  \sigma(\tau) \to \sigma(\tau) \cup i(\theta(\tau)) \leftarrow i(\theta(\tau)),
\end{equation}
which are natural in $\tau \in P_{s}(Y)$. The inclusions of (\ref{eq 2}) are identities for $\tau \in P_{s}(X)$.

For $\tau \in P_{s}(Y)$ the subset $\sigma(\tau) \cup i(\theta(\tau))$ is in $P_{s+2r}(Y)$. The inclusions of (\ref{eq 2}) define natural transformations, and hence homotopies
\begin{equation*}
  BP_{s}(Y) \times \Delta^{1} \to BP_{s+2r}(Y)
\end{equation*}
of simplicial set maps $BP_{s}(Y) \to BP_{s+2r}(Y)$.

We have proved the following:

  \begin{theorem}\label{th 4}
    Suppose $X \subset Y$ in $D(Z)$ such that $d_{H}(X,Y) < r$. Then there is a homotopy commutative diagram of poset morphisms
    \begin{equation}\label{eq 3}
  \xymatrix{
    P_{s}(X) \ar[r]^-{\sigma} \ar[d]_{i}  & P_{s+2r}(X) \ar[d]^{i} \\
    P_{s}(Y) \ar[r]_-{\sigma} \ar[ur]^(.4){\theta} & P_{s+2r}(Y)
  }
    \end{equation}
    in which the upper triangle commutes, and the lower triangle commutes up to a homotopy that fixes the subobject $P_{s}(X)$.
  \end{theorem}

  The diagram (\ref{eq 3}) in the statement of Theorem \ref{th 4} is a {\it homotopy interleaving}. Theorem \ref{th 4} is a form of the Rips stability theorem. The form of this result that appears in the Blumberg-Lesnick paper \cite{BlumLes} is the following:

  \begin{theorem}\label{th 5}
 Suppose given $X, Y \subset Z$ are data sets with $d_{H}(X,Y) < r$.
\smallskip

Then there are maps $\phi: P_{s}(X) \to P_{s+2r}(Y)$ and $\psi: P_{s}(Y) \to P_{s+2r}(X)$ such that
\begin{equation*}
\begin{aligned}
&\psi \cdot \phi \simeq \sigma : P_{s}(X) \to P_{s+4r}(X)\enskip \text{and} \\
&\phi \cdot \psi \simeq \sigma: P_{s}(Y) \to P_{s+4r}(Y).
\end{aligned}
\end{equation*}
  \end{theorem}

  Theorem \ref{th 5} is a consequence of Theorem \ref{th 4}, but it also has a poset-theoretic proof, given below, that follows the outline given by Blumberg-Lesnick \cite{BlumLes}, and uses Quillen's Theorem A \cite{Q3}, \cite[IV.5.6]{GJ}. The use of Theorem A for proofs of stability results was introduced by Memoli \cite{Memoli0}.  

  \begin{proof}[Proof of Theorem \ref{th 5}]
 Set
  \begin{equation*}
    U=\{ (x,y)\ \vert\ x \in X,\ y \in Y,\ d(x,y) < r\ \}.
  \end{equation*}
The poset $P_{s,X}(U) \subset \mathcal{P}(U)$ consists of all subsets $\sigma \subset U$ such that $d(x,x') \leq s$ for all $(x,y),(x',y') \in \sigma$. Define the poset $P_{s,Y}(U)$ similarly, by constraining distances between coordinates in $Y$.

Projection on the $X$-factor defines a poset map $p_{X}: P_{s,X}(U) \to P_{s}(X)$, and projection on the $Y$-factor defines $p_{Y}: P_{s,Y}(U) \to P_{s}(Y)$. The maps $p_{X}$ and $p_{Y}$ are weak equivalences, by Quillen's Theorem A. In effect, the slice category $p_{X}/\sigma$ can be identified with the power set of the collection of all pairs $(x,y)$ such that $x \in \sigma$, and power sets are contractible posets.

There are inclusions
\begin{equation*}
  P_{s,X}(U) \subset P_{s+2r,Y}(U),\enskip P_{s,Y}(U) \subset P_{s+2r,X}(U),
\end{equation*}
by the triangle identity, and these maps define the maps $\phi$ and $\psi$, respectvely, via the weak equivalences $p_{X}$ and $p_{Y}$.
\end{proof}

  Suppose that $X$ is a finite subset of a metric space $Z$. Write $X^{k+1}_{dis}$ for the set of $k+1$ distinct points of $X$, and think of it as a subobject of $Z^{k+1}$. The product $Z^{k+1}$ has a product metric space structure, and so we have a Hausdorff metric on its poset $D(Z^{k+1})$ of finite subsets.

  We have the following analogue (and generalization) of Theorem \ref{th 4}:

  \begin{theorem}\label{th 6}
    Suppose $X \subset Y$ in $D(Z)$ such that $d_{H}(X_{dis}^{k+1},Y_{dis}^{k+1}) < r$. Then there is a homotopy commutative diagram of poset diagrams
    \begin{equation*}
  \xymatrix{
    P_{s,k}(X) \ar[r]^-{\sigma} \ar[d]_{i}  & P_{s+2r,k}(X) \ar[d]^{i} \\
    P_{s,k}(Y) \ar[r]_-{\sigma} \ar[ur]^(.4){\theta} & P_{s+2r,k}(Y)
  }
    \end{equation*}
    in which the upper triangle commutes, and the lower triangle commutes up to a homotopy which fixes the image of $P_{s,k}(X)$.
      \end{theorem}

  \begin{proof}

    Write $P_{s,k}(X)_{0}$ for the set of one-point members (vertices) of $P_{s,k}(X)$.
    
            Suppose that $y \in P_{s,k}(Y)_{0} - P_{s,k}(X)_{0}$. Then there are $k$ points $y_{1},\dots ,y_{k}$ of $Y$, distinct from $y$ such that $d(y,y_{i}) < s$. There is a $(k+1)$-tuple $(x_{0},x_{1}, \dots ,x_{k})$ such that
            \begin{equation*}
              d((x_{0}, \dots ,x_{k}),(y,y_{1}, \dots ,y_{k})) < r,
            \end{equation*}
            by assumption. Then $d(y,x_{0}) <r$, $d(y_{i},x_{i}) < r$, and so $d(x_{0},x_{i}) < s+2r$, and $x_{0}$ is a vertex of $P_{s+2r,k}(X)$.  Set $\theta(y) = x_{0}$, and observe that $d(y,\theta(y)) < r$.

            If $[y_{0}, \dots ,y_{p}]$ is a simplex of $P_{s,k}(Y)$ then $[\theta(y_{0}), \dots ,\theta(y_{p})]$ is a simplex of $P_{s+2r,k}(Y)$, as is the subset
              \begin{equation*}
                [y_{0}, \dots ,y_{p},\theta(y_{0}), \dots ,\theta(y_{p})].
              \end{equation*}
              Finish according to the method of proof for Theorem \ref{th 4}.
\end{proof}

A data set $Y \in D(Z)$ is finite, so there is a finite string of parameter values
\begin{equation*}
  0=s_{0} < s_{1} < \dots < s_{r}, 
\end{equation*}
consisting of the distances between elements of $Y$. I say that the $s_{i}$ are the {\it phase-change numbers} for $Y$.
  
\begin{corollary}\label{cor 7} 
  Suppose that $X \subset Y$ in $D(Z)$ and that $d_{H}(X^{k+1}_{dis},Y^{k+1}_{dis}) < r$. Suppose that $2r < s_{i+1}-s_{i}$. Then the inclusion
$i: P_{s_{i},k}(X) \to P_{s_{i},k}(Y)$ is a weak homotopy equivalence.
\end{corollary}

\begin{lemma}\label{lem 8}
    Suppose that $X \subset Y$ in $D(Z)$ and that $d_{H}(X^{k+1}_{dis},Y^{k+1}_{dis}) < r$. Suppose that that $Y^{k+1}_{dis} \ne \emptyset$. Then $d_{H}(X^{k}_{dis},Y^{k}_{dis}) < r$.
\end{lemma}

\begin{proof}
    Suppose that $\{ y_{0}, \dots ,y_{k-1} \}$ is a set of $k$ distinct points of $Y$. Then there is a $y_{k} \in Y$ which is distinct from the $y_{i}$, for otherwise $Y$ has only $k$ elements. Then $(y_{0},y_{1}, \dots ,y_{k})$ is a $(k+1)$-tuple of distinct points of $Y$. There is a $(k+1)$-tuple $(x_{0},\dots ,x_{k})$ of distinct points of $X$ such that
    \begin{equation*}
      d((y_{0}, \dots ,y_{k-1},y_{k}),(x_{0}, \dots ,x_{k-1},x_{k})) < r.
      \end{equation*}
        It follows that
    \begin{equation*}
      d((y_{0}, \dots ,y_{k-1}),(x_{0}, \dots ,x_{k-1})) < r.
    \end{equation*}
\end{proof}    

\begin{corollary}
Suppose $X \subset Y$ in $D(Z)$ such that $d_{H}(X_{dis}^{k+1},Y_{dis}^{k+1}) < r$ for some non-negative number $r$. Then for $0 \leq j \leq k$ there is a homotopy commutative diagram of poset diagrams
    \begin{equation*}
  \xymatrix{
    P_{s,j}(X) \ar[r]^-{\sigma} \ar[d]_{i}  & P_{s+2r,j}(X) \ar[d]^{i} \\
    P_{s,j}(Y) \ar[r]_-{\sigma} \ar[ur]^(.4){\theta} & P_{s+2r,j}(Y)
  }
    \end{equation*}
    in which the upper triangle commutes, and the lower triangle commutes up to a homotopy that fixes the image of $P_{s,j}(X)$.
\end{corollary}

\begin{proof}
Use Theorem \ref{th 6} and Lemma \ref{lem 8}.
  \end{proof}

\section{Controlled equivalences}

A system of simplicial sets (or spaces) is a functor $X: [0,\infty) \to s\mathbf{Set}$ which takes values in the category of simplicial sets. One also says that such a functor is a diagram of simplicial sets with index category $[0,\infty)$. 
A map of systems $X \to Y$ is a natural transformation of functors that are defined on $[0,\infty)$.

We shall also discuss systems of sets, groups and chain complexes as functors defined on the poset $[0,\infty)$, which take values in the respective categories.
\medskip

\noindent
{\bf Examples}
\smallskip

\noindent
1)\ The functors $s \mapsto V_{s}(X), BP_{s}(X)$ are systems of spaces, for a data set $X \subset Z$. The functor $s \mapsto P_{s}(X)$ is a system of posets.
\smallskip

\noindent
2)\ If $X \subset Y \subset Z$ are data sets, the induced maps $P_{s}(X) \to P_{s}(Y)$ and $V_{s}(X) \to V_{s}(Y)$ define maps of systems $P_{\ast}(X) \to P_{\ast}(Y)$ and $V_{\ast}(X) \to V_{\ast}(Y)$.
\medskip

There are various ways to discuss homotopy theories of systems. The oldest of these is the projective model structure of Bousfield and Kan \cite{BK}, although they do not use the term ``projective'' --- this term arose much later in motivic homotopy theory.

In the projective structure, a map $f: X \to Y$ is a weak equivalence (respectively fibration) if each map $X_{s} \to Y_{s}$ is a weak equivalence (respectively fibration) of simplicial sets.
 The maps which are both weak equivalences and fibrations are called trivial fibrations.

A map $A \to B$ of systems is a projective cofibration  if it has the ``left lifting property'' with respect all maps which are weak equivalences and fibrations. Projective cofibrations are intensely studied and important, but will not be used here.

A map of systems $f:X \to Y$ such that each map $f:X_{s} \to Y_{s}$ is a weak equivalence (respectively fibration) is often said to be a sectionwise weak equivalence (respectively sectionwise fibration). A sectionwise cofibration is a map of systems $A \to B$ such the each map $A_{s} \to B_{s}$ is a monomorphism (or cofibration) of simplicial sets. Every projective cofibration is a sectionwise cofibration, but the converse is not true.
\medskip

Suppose that $X \subset Y$ in $D(Z)$ such that $d_{H}(X,Y) < r$. The Rips stability theorem (Theorem \ref{th 4}) says that we have a homotopy interleaving
    \begin{equation*}
  \xymatrix{
    BP_{s}(X) \ar[r]^-{\sigma} \ar[d]_{i}  & BP_{s+2r}(X) \ar[d]^{i} \\
    BP_{s}(Y) \ar[r]_-{\sigma} \ar[ur]^(.4){\theta} & BP_{s+2r}(Y)
  }
\end{equation*}
where the upper triangle commutes and the lower triangle commutes up to homotopy which is constant on the space $BP_{s}(X)$.

 Then we have the following:
 \begin{itemize}
   \item[1)] The natural transformation $i: \pi_{0}BP_{\ast}(X) \to \pi_{0}BP_{\ast}(Y)$ is a {\it $2r$-mono\-morphism}: if $i([x]) = i([y])$ in $\pi_{0}BP_{s}(Y)$ then $\sigma[x]=\sigma[y]$ in $\pi_{0}BP_{s+2r}(X)$.

\item[2)] The transformation $i: \pi_{0}BP_{\ast}(X) \to \pi_{0}BP_{\ast}(Y)$ is a {\it $2r$-epimorphism}: given $[y] \in \pi_{0}BP_{s}(Y)$, $\sigma[y]=i[x]$ for some $[x] \in \pi_{0}BP_{s+2r}(X)$.

\item[3)] All natural transformations $i: \pi_{n}(BP_{\ast}(X), x) \to \pi_{n}(BP_{\ast}(Y),i(x))$ of homotopy group functors are {\it $2r$-isomorphisms} in the sense that they are both $2r$-monomorphisms and $2r$-epimorphisms.
\end{itemize}
 \medskip

     The statements 1)--3) are ``derived'', and depend on having a way to talk about higher homotopy groups.

     There is a functorial weak equivalence $\gamma: X \to \Ex^{\infty}X$ of simplicial sets, where $\Ex^{\infty}X$ is a system of Kan complexes, and therefore have combinatorially defined homotopy groups \cite{GJ}. Thus, for example, the notation $\pi_{n}(BP_{s}(X),x)$ can refer to the combinatorial homotopy group $\pi_{n}(\Ex^{\infty}BP_{s}(X),x)$.

     There is an alternative, in that one could use the adjunction weak equivalence $\eta: X \to S(\vert X \vert)$, where $S$ is the singular functor and $\vert X \vert$ is the topological realization of $X$. The combinatorial homotopy groups of the Kan complex $S(\vert X \vert)$ coincide up to natural isomorphism with the standard homotopy groups of the space $\vert X \vert$. In this case, we would write $\pi_{n}(BP_{s}(X),x)$ to mean $\pi_{n}(\vert BP_{s}(X)\vert,x)$.

     There is a natural isomorphism
     \begin{equation*}
       \pi_{n}(\Ex^{\infty}Y,y) \cong \pi_{n}(\vert Y \vert,y)
     \end{equation*}
     for all simplicial sets $Y$ and vertices $y$ of $Y$, so the combinatorial and topological constructions produce isomorphic homotopy groups.

     The Kan $\Ex^{\infty}$ functor is combinatorial and therefore plays well with algebraic constructions, while the realization functor is familiar but transcendental.

     The homotopy groups $\pi_{n}(BP_{s}(X),x)$ coincide with the homotopy groups $\pi_{n}(V_{s}(X),x)$ of the Vietoris-Rips complex $V_{s}(X)$ up to natural isomorphism. A similar observation holds for the extant constructions of the degree Rips complexes.

The natural maps $\gamma: Y \to \Ex^{\infty}Y$ and $\eta: Y \to S(\vert Y \vert)$ are fibrant models for simplicial sets $Y$, in that the maps are weak equivalences which take values in fibrant simplicial sets (Kan complexes). Both constructions preserve monomorphisms. 

We shall write $Y \to FY$ for an arbitrary fibrant model construction that preserves monomorphisms. A formal nonsense argument implies that the choice of fibrant model does not matter.
\medskip

     Suppose $f: X \to Y$ is a map of systems. Say that $f$ is an {\it $r$-equivalence} if
\begin{itemize}
\item[1)] the map $f: \pi_{0}(X) \to \pi_{0}(Y)$ is an $r$-isomorphism of systems of sets
\item[2)]
the maps $f: \pi_{k}(X_{t},x) \to \pi_{k}(Y_{t}.f(x))$ are $r$-isomorphisms of systems of groups for $t \geq s$, for all $s \geq 0$ and $x \in X_{s}$.
\end{itemize}

\begin{remark}
Note the variation of condition 2) from the Vietoris-Rips example. In the general definition, we do not assume that all simplicial sets $X_{s}$ of the system $X$ have the same vertices, so the base points of condition 2) have to be chosen section by section.
This is relevant for comparisons of degree Rips systems $BP_{\ast,k}(X) \to BP_{\ast,k}(Y)$.
\end{remark}

\noindent
{\bf Examples}:\ 1)\ A map $f: X \to Y$ is a sectionwise equivalence if and only if it is a $0$-equivalence.
\medskip

\noindent
2)\ If $r \leq s$ and $f: X \to Y$ is an $r$-equivalence, then $f$ is an $s$-equivalence.
\medskip

\noindent
3)\
If $X \subset Y$ are (finite) data sets in a metric space $Z$, then there is a homotopy $r$-interleaving 
    \begin{equation*}
  \xymatrix{
    BP_{s}(X) \ar[r]^-{\sigma} \ar[d]_{i}  & BP_{s+r}(X) \ar[d]^{i} \\
    BP_{s}(Y) \ar[r]_-{\sigma} \ar[ur]^(.4){\theta} & BP_{s+r}(Y)
  }
\end{equation*}
    for sufficiently large $r$: take $r/2 > d_{H}(X,Y)$. This means that the map of systems $BP_{s}(X) \to BP_{s}(Y)$ is an $r$-equivalence for large $r$.

    A similar observation holds for the comparison  $BP_{s,k}(X) \to BP_{s,k}(Y)$ of degree Rips systems.
    \medskip

    Say that a map of systems $X \to Y$ is a {\it controlled equivalence} if it is an $r$-equivalence for some $r \geq 0$. 
    \medskip

The class of controlled equivalences of systems has a list of formal properties, which begins with the following result.

\begin{lemma}\label{lem 11}
    Suppose given a diagram of systems
\begin{equation*}
\xymatrix{
X_{1} \ar[r]^{f_{1}} \ar[d]_{\simeq} & Y_{1} \ar[d]^{\simeq} \\
X_{2} \ar[r]_{f_{2}} & Y_{2}
}
\end{equation*}
in which the vertical maps are sectionwise weak equivalences.
Then $f_{1}$ is an $r$-equivalence if and only if $f_{2}$ is an $r$-equivalence.
\end{lemma}

\begin{lemma}\label{lem 12}
Suppose given a commutative triangle
\begin{equation*}
\xymatrix{
X \ar[r]^{f} \ar[dr]_{h} & Y \ar[d]^{g} \\
& Z
}
\end{equation*}
of maps of systems,'

Then if one of the maps is an $r$-equivalence, a second is an $s$-equivalence, then the third map is a $(r+s)$-equivalence.
\end{lemma}

\begin{proof}
  The arguments are set theoretic. We present an example.
  
Suppose $X,Y,Z$ are systems of sets, $h$ is an $r$-isomorphism and $g$ is an $s$-isomorphism. Given $z \in Y_{t}$, $g(z) = h(w)$ for some $w \in X_{t+s}$. Then $g(z) = g(f(w))$ in $Z_{t+s}$ so $z = f(w)$ in $Y_{t+s+r}$. It follows that $f$ is an $(r+s)$-epimorphism.
\end{proof}

Lemma \ref{lem 12} is an approximation of the triangle axiom for weak equivalences in the definition of a Quillen model structure.
\medskip

There is a calculus of controlled equivalences and sectionwise fibrations, which starts with the following result and concludes with Theorem \ref{th 15}.

 \begin{lemma}\label{lem 13}
  Suppose that $p: X \to Y$ is a sectionwise fibration of systems of Kan complexes, and that $p$ is an $r$-equivalence.

  Then each lifting problem
\begin{equation*}
  \xymatrix{
    \partial\Delta^{n} \ar[r]^{\alpha} \ar[d] & X_{s} \ar[d] \ar[r]^{\sigma}
    & X_{s+2r} \ar[d]^{p} \\
    \Delta^{n} \ar[r]_{\beta} \ar@{.>}[urr]^(.4){\theta} & Y_{s} \ar[r]_{\sigma} & Y_{s+2r}
  }
\end{equation*}
can be solved up to shift $2r$ in the sense that the indicated dotted arrow lifting exists.
 \end{lemma}

 Lemma \ref{lem 13} is the analogue of a classical result of simplicial homotopy theory \cite[I.7.10]{GJ}, and its proof is a variant of the standard obstruction theoretic argument for that result.

\begin{proof}[Proof of Lemma \ref{lem 13}]
The original diagram can be replaced up to homotopy by a diagram
\begin{equation}\label{eq 4}
\xymatrix@C=40pt{
\partial\Delta^{n} \ar[r]^{(\alpha_{0},\ast,\dots ,\ast)} \ar[d] & X_{s} \ar[d]^{p}
\ar[r]^{\sigma}
& X_{s+r} \ar[d]^{p} \\
\Delta^{n} \ar[r]_{\beta}& Y_{s} \ar[r]_{\sigma} & Y_{s+r}
}
\end{equation}
Then $p_{\ast}([\alpha_{0}]) = 0$ in $\pi_{n-1}(Y_{s},\ast)$, so $\sigma_{\ast}([\alpha_{0}]) = 0 $ in $\pi_{n-1}(X_{s+r},\ast)$.
\smallskip

The trivializing homotopy for $\sigma(\alpha_{0})$ in $X_{s+r}$ defines a homotopy from the outer square of (\ref{eq 4}) to the diagram
\begin{equation*}
\xymatrix{
\partial\Delta^{n} \ar[r]^{\ast} \ar[d] & X_{s+r} \ar[d]^{p} \\
\Delta^{n} \ar[r]_{\omega} & Y_{s+r}
}
\end{equation*}

The element $[\omega] \in \pi_{n}(Y_{s+2r},\ast)$ lifts to an element of $\pi_{n}(X_{s+2r},\ast)$ up to homotopy, giving the desired lifting.
\end{proof}

 \begin{lemma}\label{lem 14}
Suppose that $p: X \to Y$ is a sectionwise fibration of systems of Kan complexes, and that all lifting problems
\begin{equation*}
  \xymatrix{
    \partial\Delta^{n} \ar[r] \ar[d] & X_{s} \ar[d] \ar[r]^{\sigma}
    & X_{s+r} \ar[d]^{p} \\
    \Delta^{n} \ar[r] \ar@{.>}[urr]^(.4){\theta} & Y_{s} \ar[r]_{\sigma} & Y_{s+r}
  }
\end{equation*}
have solutions up to shift $r$, in the sense that the dotted arrow exists making the diagram commute. Then the map $p: X \to Y$ is an $r$-equivalence.
   \end{lemma}

\begin{proof}
If $p_{\ast}([\alpha]) =0$ for $[\alpha] \in \pi_{n-1}(X_{s},\ast)$, then there is a diagram on the left above. The existence of $\theta$ gives $\sigma_{\ast}([\alpha]) = 0$ in $\pi_{n-1}(X_{s+r},\ast)$.

The argument for $r$-surjectivity is similar.
\end{proof}

Lemma \ref{lem 13} and Lemma \ref{lem 14} together imply the following:

\begin{theorem}\label{th 15}
  Suppose given a pullback diagram
  \begin{equation*}
    \xymatrix{
      X' \ar[r] \ar[d]_{p'} & X \ar[d]^{p} \\
    Y' \ar[r] & Y
  }
  \end{equation*}
 where $p$ is a sectionwise fibration and an $r$-equivalence.

  Then the map $p'$ is a sectionwise fibration and a $2r$-equivalence.
\end{theorem}

\begin{proof}
  All lifting problems
  \begin{equation*}
  \xymatrix{
    \partial\Delta^{n} \ar[r]^{\alpha} \ar[d] & X'_{s} \ar[d] \ar[r]^{\sigma}
    & X'_{s+2r} \ar[d]^{p'} \\
    \Delta^{n} \ar[r]_{\beta} \ar@{.>}[urr]^(.4){\theta} & Y'_{s} \ar[r]_{\sigma} & Y'_{s+2r}
  }
\end{equation*}
for $p'$ have solutions up to shift $2r$, since it is a pullback of a map $p$ that has that property by Lemma \ref{lem 13}. Then Lemma \ref{lem 14} implies that $p'$ is a $2r$-equivalence.
  \end{proof}

Suppose that $i: A \to B$ is a sectionwise cofibration of projective cofibrant systems (i.e. systems of monomorphisms), and form the diagram
\begin{equation*}
  \xymatrix{
    A \ar[r]^{\eta} \ar[d]_{i} & FA \ar[d]^{i_{\ast}} \\
    B \ar[r]_{\eta} & FB
  }
\end{equation*}
in which the horizontal maps are fibrant models, and in particular sectionwise equivalences. Typically, one sets $FA = \Ex^{\infty}A$.

Say that the map $i$ {\it admits an $r$-interleaving} if, for all $s$, there are maps $\theta: B_{s} \to FA_{s+r}$ such that the diagram
\begin{equation*}
  \xymatrix{
    A_{s} \ar[r]^{\sigma} \ar[d]_{i} & A_{s+r} \ar[r]^{\eta} & FA_{s+r} \\
    B_{s} \ar[urr]_{\theta}
  }
\end{equation*}
commutes, and the diagram
\begin{equation*}
  \xymatrix{
    && FA_{s+r} \ar[d]^{i_{\ast}} \\
    B_{s} \ar[urr]^{\theta} \ar[r]_{\sigma} & B_{s+r} \ar[r]_{\eta} & FB_{s+r}
  }
\end{equation*}
commutes up to a homotopy which restricts to the constant homotopy on $A_{s}$.
\medskip

An $r$-interleaving is effectively a strong deformation retraction up to shift $r$.

\begin{lemma}\label{lem 16}
Suppose that the map $i: A \to B$ admits an $r$-interleaving. Then the map $i: A \to B$ is an $r$-equivalence.
\end{lemma}

\begin{proof}
  The map $\eta: A_{s+r} \to FA_{s+r}$ is a weak equivalence, so that $\theta: B_{s} \to FA_{s+r}$ induces functions $\theta_{\ast}: \pi_{0}B_{s} \to \pi_{0}A_{s+r}$ and homomorphisms
\begin{equation*}
  \theta_{\ast}: \pi_{n}(B_{s},i(x)) \to \pi_{n}(A_{s+r},\sigma(x)).
  \end{equation*}

  The diagram
  \begin{equation*}
    \xymatrix{
      \pi_{0}A_{s} \ar[r]^-{\sigma} \ar[d]_{i} & \pi_{0}A_{s+r} \ar[d]^{i} \\
      \pi_{0}B_{s} \ar[r]_-{\sigma} \ar[ur]^{\theta} & \pi_{0}B_{s+r}
    }
    \end{equation*}
  commutes, so that the map of systems of sets $\pi_{0}A \to \pi_{0}B$ is an $r$-isomorphism.

  The diagram
  \begin{equation*}
    \xymatrix{
      \pi_{n}(A_{s},x) \ar[r]^-{\sigma} \ar[d]_{i}
      & \pi_{n}(A_{s+r},\sigma(x)) \ar[d]^{i} \\
      \pi_{n}(B_{s},i(x)) \ar[ur]^{\theta} \ar[r]_-{\sigma} & \pi_{n}(B_{s+r},\sigma i(x))
    }
  \end{equation*}
  also commutes. In effect, $\sigma(i(x)) = i(\theta(i(x))$, and the homotopy $\sigma \simeq i \cdot \theta$ restricts to the identity on $i(x)$, so that $[\sigma(\alpha)] = [i(\theta(\alpha)]$ in $\pi_{n}(B_{s+r},\sigma i(x))$ for any representing simplex $\alpha: \Delta^{n} \to FB_{s}$ of a homotopy group element $[\alpha] \in \pi_{n}(B_{s},i(x))$.  
  \end{proof}

\begin{lemma}\label{lem 17}
  Suppose that $i: A \to B$ is a sectionwise cofibration between systems, such that $i$ admits an $r$-interleaving. Suppose also that the diagram
  \begin{equation}\label{eq 5}
    \xymatrix{
      A \ar[r]^{\alpha} \ar[d]_{i} & C \ar[d]^{i'} \\
      B \ar[r]_{\beta} & D
    }
  \end{equation}
  is a pushout. Then the map $i'$ admits an $r$-interleaving.  
\end{lemma}

\begin{proof}
  The composites
  \begin{equation*}
    \begin{aligned}
      &B_{s} \xrightarrow{\theta} FA_{s+r} \xrightarrow{\alpha_{\ast}} FC_{s+r},  \\
      &C_{s} \xrightarrow{\sigma} C_{s+r} \xrightarrow{\eta} FC_{s+r}
    \end{aligned}
    \end{equation*}
      together determine a unique map $\theta': D_{s} \to FC_{s+r}$.

      The homotopy $i_{\ast}\cdot \theta \simeq \eta \cdot \sigma$ is defined by a map
      \begin{equation*}
        h: B_{s} \times \Delta^{1} \to FB_{s+r}.
      \end{equation*}
      This homotopy restricts to a constant homotopy on $A_{s}$, 
  which means that the diagram
      \begin{equation*}
        \xymatrix{
          A_{s} \times \Delta^{1} \ar[r]^-{pr} \ar[d]_{i \times \Delta^{1}} & A_{s} \ar[r]^{\sigma} & A_{s+r} \ar[r]^{\eta}
          & FA_{s+r} \ar[d]^{i_{\ast}} \\
          B_{s} \times \Delta^{1} \ar[rrr]_{h} &&& FB_{s+r}
        }
      \end{equation*}
      commutes, where $pr$ is a projection.

      The diagram
      \begin{equation*}
        \xymatrix{
          A_{s} \times \Delta^{1} \ar[r]^{\alpha \times \Delta^{1}} \ar[d]_{i \times \Delta^{1}}
          & C_{s} \times \Delta^{1} \ar[d]^{i_{\ast} \times \Delta^{1}} \\
          B_{s} \times \Delta^{1} \ar[r]_{\beta \times \Delta^{1}} & D_{s} \times \Delta^{1}
        }
      \end{equation*}
      is a pushout, and the composites
      \begin{equation*}
        \begin{aligned}
          &B_{s} \times \Delta^{1} \xrightarrow{h} FB_{s+r} \xrightarrow{\beta_{\ast}} FD_{s+r} \\
          &C_{s} \times \Delta^{1} \xrightarrow{pr} C_{s} \xrightarrow{\sigma} C_{s+r} \xrightarrow{\eta} FC_{s+r} \xrightarrow{i_{\ast}} FD_{s+r}
        \end{aligned}
      \end{equation*}
      together determine a homotopy
      \begin{equation*}
        h': D_{s} \times \Delta^{1} \to FD_{s+r}
      \end{equation*}
      from $i_{\ast} \cdot \theta'$ to $\eta \cdot \sigma$. The homotopy $h'$ restricts to the constant homotopy on $C_{s}$,
              by construction.
\end{proof}

Theorem \ref{th 15} says that the pullback of a map which is a sectionwise fibration and an $r$-equivalence is a sectionwise fibration and a $2r$-equivalence. The ``dual'' statement, namely that a pushout of a map which is a cofibration and an $r$-equivalence is a cofibration and a $2r$-equivalence, has not been proved for any relevant class of cofibrations.

The practical examples of cofibrations which are $r$-weak equivalences are cofibrations which admit $r$-interleavings. These arise from stability results, and Lemma \ref{lem 17} says that the class of cofibrations which admit $r$-interleavings is closed under pushout.

Thus, if one has a pushout diagram such as (\ref{eq 5}) for which the cofibration $i$ admits an $r$-interleaving, then the induced map $i_{\ast}$ also admits an $r$-interleaving. The map $i_{\ast}$ induces $r$-isomorphisms $H_{k}(C) \to H_{k}(D)$ and $\pi_{0}C \to \pi_{0}D$.
\medskip

We have more general statements for homology and path components, as follows.

\begin{lemma}\label{lem 18}
  Supppose given a pushout diagram
  \begin{equation*}
    \xymatrix{
      A \ar[r]^{\alpha} \ar[d]_{i} & C \ar[d] \\
      B \ar[r]_{\beta} & D
    }
  \end{equation*}
  of systems in which $i$ is a cofibration. Then the following hold:
  \begin{itemize}
  \item[1)]
    If the map $\pi_{0}A \to \pi_{0}B$ is an $r$-isomorphism, then the $\pi_{0}C \to \pi_{0}D$ is an $r$-isomorphism.
  \item[2)] If the maps $H_{k}(A) \to H_{k}(B)$ are $r$-isomorphisms for $k \geq 0$ (arbitrary coefficients, then the map $H_{k}(C) \to H_{k}(D)$ is a $2r$-isomorphism, for $k \geq 0$.
\end{itemize}
  \end{lemma}

\begin{proof}
  Statement 1) follows from Lemma \ref{lem 19} below, since the path component functor preserves pushouts.

  For statement 2) there is a system of exact sequences
  \begin{equation*}
    \dots \to H_{k}(A) \to H_{k}(B) \to H_{k}(B/A) \xrightarrow{\partial} H_{k-1}(A) \to H_{k-1}(B) \to \dots
  \end{equation*}
  An element chase within this system shows that the map
  $0 \to H_{k}(B/A)$ is a $2r$-isomorphism for all $k \geq 0$. One uses the system of exact sequences
\begin{equation*}
  \dots \to H_{k}(C) \to H_{k}(D) \to H_{k}(B/A) \xrightarrow{\partial} H_{k-1}(C) \to H_{k-1}(DB) \to \dots
  \end{equation*}
 to show that the map $H_{k}(C) \to H_{k}(D)$ is a $2r$-isomorphism for all $k \geq 0$.
  \end{proof}

  \begin{lemma}\label{lem 19}
  Suppose that the diagram
  \begin{equation*}
    \xymatrix{
    A \ar[r]^{\alpha} \ar[d]_{f} & C \ar[d]^{f'} \\
    B \ar[r] & D
    }
  \end{equation*}
  is a pushout of systems of sets, and that $f$ is an $r$-bijection. Then $f'$ is an $r$-bijection.
\end{lemma}

\begin{proof}
  The map $f$ has an epi-monic factorization
  \begin{equation*}
    A \xrightarrow{p} Z \xrightarrow{j} B
  \end{equation*}
  and there are corresponding pushout diagrams
  \begin{equation*}
    \xymatrix@R=12pt{
      A \ar[r] \ar[d]_{p} & C \ar[d]^{p'} \\
      Z \ar[d]_{j} \ar[r] & Z' \ar[d]^{j'} \\
      B \ar[r] & D
    }
    \end{equation*}

  The map $j'$ is a sectionwise monomorphism, and $j'$ is an $r$-epimorphism since $j$ is an $r$-epimorphism.

  The sectionwise epimorphism $p'$ is constructed by collapsing images of fibres of $p$ to points, and it follows that $p'$ is an $r$-monomorphism as well as a sectionwise epimorphism.

The conclusion follows: the composite $j' \cdot p'$ is an $r$-monomorphism and an $r$-epimorphism.
\end{proof}

\bibliographystyle{plain}
\bibliography{spt}

\end{document}